\documentclass[11pt]{amsart}
\usepackage{amsmath}
\usepackage{amsthm}
\usepackage{amssymb}
\usepackage{verbatim}
\usepackage{hyperref}
\usepackage{color}
\usepackage{cite}

\begin{document}

\newtheorem{theorem}{Theorem}    
\newtheorem{proposition}[theorem]{Proposition}
\newtheorem{conjecture}[theorem]{Conjecture}
\def\theconjecture{\unskip}
\newtheorem{corollary}[theorem]{Corollary}
\newtheorem{lemma}[theorem]{Lemma}
\newtheorem{sublemma}[theorem]{Sublemma}
\newtheorem{observation}[theorem]{Observation}
\newtheorem{remark}[theorem]{Remark}
\newtheorem{definition}[theorem]{Definition}
\theoremstyle{definition}
\newtheorem{notation}[theorem]{Notation}
\newtheorem{question}[theorem]{Question}
\newtheorem{questions}[theorem]{Questions}
\newtheorem{example}[theorem]{Example}
\newtheorem{problem}[theorem]{Problem}
\newtheorem{exercise}[theorem]{Exercise}

\numberwithin{theorem}{section} \numberwithin{theorem}{section}
\numberwithin{equation}{section}

\def\earrow{{\mathbf e}}
\def\rarrow{{\mathbf r}}
\def\uarrow{{\mathbf u}}
\def\varrow{{\mathbf V}}
\def\tpar{T_{\rm par}}
\def\apar{A_{\rm par}}

\def\reals{{\mathbb R}}
\def\torus{{\mathbb T}}
\def\heis{{\mathbb H}}
\def\integers{{\mathbb Z}}
\def\naturals{{\mathbb N}}
\def\complex{{\mathbb C}\/}
\def\distance{\operatorname{distance}\,}
\def\support{\operatorname{support}\,}
\def\dist{\operatorname{dist}\,}
\def\Span{\operatorname{span}\,}
\def\degree{\operatorname{degree}\,}
\def\kernel{\operatorname{kernel}\,}
\def\dim{\operatorname{dim}\,}
\def\codim{\operatorname{codim}}
\def\trace{\operatorname{trace\,}}
\def\Span{\operatorname{span}\,}
\def\dimension{\operatorname{dimension}\,}
\def\codimension{\operatorname{codimension}\,}
\def\nullspace{\scriptk}
\def\kernel{\operatorname{Ker}}
\def\ZZ{ {\mathbb Z} }
\def\p{\partial}
\def\rp{{ ^{-1} }}
\def\Re{\operatorname{Re\,} }
\def\Im{\operatorname{Im\,} }
\def\ov{\overline}
\def\eps{\varepsilon}
\def\lt{L^2}
\def\diver{\operatorname{div}}
\def\curl{\operatorname{curl}}
\def\etta{\eta}
\newcommand{\norm}[1]{ \|  #1 \|}
\def\expect{\mathbb E}
\def\bull{$\bullet$\ }
\def\C{\mathbb{C}}
\def\R{\mathbb{R}}
\def\Rn{{\mathbb{R}^n}}
\def\Sn{{{S}^{n-1}}}
\def\M{\mathbb{M}}
\def\N{\mathbb{N}}
\def\Q{{\mathbb{Q}}}
\def\Z{\mathbb{Z}}
\def\F{\mathcal{F}}
\def\L{\mathcal{L}}
\def\S{\mathcal{S}}
\def\supp{\operatorname{supp}}
\def\dist{\operatorname{dist}}
\def\essi{\operatornamewithlimits{ess\,inf}}
\def\esss{\operatornamewithlimits{ess\,sup}}
\def\xone{x_1}
\def\xtwo{x_2}
\def\xq{x_2+x_1^2}
\newcommand{\abr}[1]{ \langle  #1 \rangle}

\newcommand{\Norm}[1]{ \left\|  #1 \right\| }
\newcommand{\set}[1]{ \left\{ #1 \right\} }
\def\one{\mathbf 1}
\def\whole{\mathbf V}
\newcommand{\modulo}[2]{[#1]_{#2}}
\renewcommand{\thefootnote}{\fnsymbol{footnote}}
	\def\scriptf{{\mathcal F}}
	\def\scriptg{{\mathcal G}}
	\def\scriptm{{\mathcal M}}
	\def\scriptb{{\mathcal B}}
	\def\scriptc{{\mathcal C}}
	\def\scriptt{{\mathcal T}}
	\def\scripti{{\mathcal I}}
	\def\scripte{{\mathcal E}}
	\def\scriptv{{\mathcal V}}
	\def\scriptw{{\mathcal W}}
	\def\scriptu{{\mathcal U}}
	\def\scriptS{{\mathcal S}}
	\def\scripta{{\mathcal A}}
	\def\scriptr{{\mathcal R}}
	\def\scripto{{\mathcal O}}
	\def\scripth{{\mathcal H}}
	\def\scriptd{{\mathcal D}}
	\def\scriptl{{\mathcal L}}
	\def\scriptn{{\mathcal N}}
	\def\scriptp{{\mathcal P}}
	\def\scriptk{{\mathcal K}}
	\def\frakv{{\mathfrak V}}

	\allowdisplaybreaks
	
	\arraycolsep=1pt
	%
	\newtheorem*{remark0}{\indent\sc Remark}
	%
	\renewcommand{\proofname}{Proof.} 
	
\title[BMO-Type characterizations and endpoint estimates]
{BMO-Type characterizations and endpoint estimates for commutators of
intrinsic square functions on Orlicz--Hardy spaces}

	\author[Yanyan Han]{Yanyan Han}
\address{Yanyan Han: School of Information Network Security\\People's Public Security
         University of China\\ Beijing 100038\\ People's Republic of China}
\email{hanyanyan$\_$bj@163.com}

\author[Feng Liu]{Feng Liu$^*$}
\address{Feng Liu:
        College of Mathematics and System Science\\
        Shandong University of Science and Technology\\
        Qingdao, Shandong 266590\\
        People's Republic of China}
\email{FLiu@sdust.edu.cn}

\author[Yongming Wen]{Yongming Wen}
\address{Yongming Wen:
       School of Mathematics and Statistics\\Minnan Normal University\\ Zhangzhou 363000\\
         People's Republic of China}
\email{wenyongmingxmu@163.com}
	
	\author[Huoxiong Wu]{Huoxiong Wu}
	\address{Huoxiong Wu: School of Mathematical Sciences\\ Xiamen University\\
		Xiamen 361005\\	People's Republic of China}
	\email{huoxwu@xmu.edu.cn}

	\keywords{Intrinsic square function, commutator, Musielak--Orlicz Hardy spaces, BMO-type spaces, endpoint estimates, Orlicz--Hardy spaces.\\
		\thanks{$^*$Corresponding author.}
		\indent{2020 Mathematics Subject Classification.} 47B47, 42B20, 42B30, 42B35.}


\begin{abstract}

The commutators of intrinsic square functions associated with the intrinsic Littlewood--Paley $g$-function, the intrinsic $g_{\lambda}^{*}$-function and the intrinsic Lusin area function can be used to characterize BMO-type function spaces through their boundedness properties. In this paper, we show that, for $b\in {\rm BMO}(\mathbb{R}^{n})$, the boundedness of each of these commutators from the Musielak--Orlicz Hardy space $H^{\varphi}(\mathbb{R}^{n})$ to $L^{\varphi}(\mathbb{R}^{n})$ is equivalent to $b\in {\rm BMO}_{\varphi}(\mathbb{R}^{n})$ (a nontrivial subspace of $\rm{BMO}(\mathbb{R}^{n})$), under suitable assumptions on the growth function $\varphi$. This extends the weighted characterization of Han and Wu [Proc. Amer. Math. Soc. 152(1) (2024), 281--293] to the Musielak--Orlicz setting.  In addition, under suitable assumptions on a growth function $\Phi$, we prove that, for $b\in {\rm BMO}_{\Phi}(\mathbb{R}^{n})$, these commutators are bounded from the $b$-adapted Orlicz--Hardy space $H_{b}^{\Phi}(\mathbb{R}^{n})$ to $L^{1}(\mathbb{R}^{n})$ and from Orlicz Hardy space $H^{\Phi}(\mathbb{R}^{n})$ to $L^{1,\infty}(\mathbb{R}^{n})$.
\end{abstract}

\maketitle

\section{Introduction}

In this paper we are concerned with the BMO-type function space
characterizations and endpoint estimates for commutators of intrinsic square
functions on Orlicz--Hardy spaces.
It is well known that square functions are fundamental objects in Littlewoo--Paley theory, a cornerstone of the harmonic analysis of function spaces, which provides a comprehensive and unifying perspective on norm inequalities and function-space characterizations; see \cite{CY,CXY,WB,SXY,XYY} and the references therein.
Recently, the boundedness of commutators associated with intrinsic square functions on various function spaces has also been investigated by many authors; see \cite{HW,Wang2012,
GOS,LNYZ}.
In particular, the study of BMO-type function space characterizations via the boundedness of commutators of various classical operators on Hardy spaces has been the subject of many recent articles (see \cite{HK,HW,LKY}). Inspired by the results in \cite{HK,HW}, this paper is devoted to providing some new characterizations of a class of ${\rm BMO}$-type function spaces introduced by Huy and Ky \cite{HK} via the boundedness of commutators of
intrinsic square functions including the intrinsic Littlewood--Paley $g$-function,
the intrinsic $g_{\lambda}^{*}$-function and the intrinsic Lusin area function.
In addition, we establish new endpoint estimates for these commutators on Orlicz--Hardy spaces.

\subsection{BMO-Type characterizations of commutators of intrinsic square functions
on Hardy spaces}

The intrinsic square function originates from Wilson \cite{W} who first
studied the weighted $L^p$ boundedness. These operators not only
play a key role in the study of a conjecture of Fefferman and Stein
concerning a weighted $L^{2}$ estimate for the Lusin area function, but also dominate the classical Littlewood--Paley square functions pointwise (see \cite{WB}). For $0<\alpha\leq1$, let $\mathcal{C}_\alpha$
be the family of functions $\phi$ defined on $\mathbb{R}^n$ such that
$\supp \phi\subset B(0,1),$ $\int_{\mathbb{R}^n}\phi(x)dx=0$ and for
all $x,\,x'\in \mathbb{R}^n$,
$$|\phi(x)-\phi(x')|\leq |x-x'|^\alpha.$$
For $(y,t)\in \mathbb{R}^{n+1}_+$ and $f\in L_{\rm loc}(\mathbb{R}^n)$
we set
$$A_\alpha(f)(y,t)=\sup_{\phi\in \mathcal C_\alpha}|f*\varphi_t(y)|.$$
The intrinsic Lusin area function is defined by
$$S_\alpha(f)(x):=\Big(\iint_{\Gamma(x)}[A_\alpha(f)(y,t)]^2\frac{dy\,dt}{t^{n+1}}\Big)^{1/2},$$
where $\Gamma(x):=\Big\{(y,t)\in\mathbb{R}^{n+1}_+:\ |x-y|<t\Big\}$. The
intrinsic Littlewood--Paley $g$-function is defined by
$$g_\alpha(f)(x):=\Big(\int_0^\infty[A_\alpha(f)(x,t)]^2\frac{dt}{t}\Big)^{1/2},$$
while the intrinsic Littlewood--Paley $g_{\lambda}^*$-function is given by
$$g_{\lambda,\alpha}^*(f)(x):=\left(\iint_{\mathbb{R}^{n+1}_+}\left(\frac{t}{t+|x-y|}\right)^{\lambda n}[A_\alpha(f)(y,t)]^2\frac{dy\,dt}{t^{n+1}}\right)^{1/2}.$$
Over the last several years the mapping properties of the above intrinsic
square functions on various function spaces have been studied by many authors.
For example, see \cite{Lerner,WL} for weighted Lebesgue spaces,
\cite{HL,Wang2012,GOS,LY} for weighted Hardy spaces, Morrey spaces and
Musielak--Orlicz Morrey spaces. Meanwhile, the intrinsic square functions
also play a key role in characterizing various Hardy-type spaces (see
\cite{HL,LY,WL}).

The commutators of the intrinsic Lusin area function, the intrinsic Littlewood--Paley $g$-function, and the intrinsic Littlewood--Paley $g_{\lambda}^{*}$-function are defined, respectively, by
$$S_{\alpha,b}(f)(x):=\Big(\iint_{\Gamma(x)}\sup_{\phi\in\mathcal{C}_\alpha}
\Big|\int_{\mathbb{R}^n}[b(x)-b(z)]\phi_t(y-z)f(z)\,dz\Big|^2\frac{dy\,dt}{t^{n+1}}
\Big)^{1/2},$$
$$g_{\alpha,b}(f)(x):=\Bigg(\int_0^\infty\sup_{\phi\in\mathcal{C}_\alpha}\Big|
\int_{\mathbb{R}^n}[b(x)-b(y)]\phi_t(x-y)f(y)\,dy\Big|^2\frac{dt}{t}\Bigg)^{1/2},$$
and
$$g_{\alpha,b}^{\lambda,*}(f)(x):=\Bigg(\iint_{\mathbb{R}^{n+1}_+}(\frac{t}{t+|x-y|})^{\lambda n}\sup_{\phi\in\mathcal{C}_\alpha}\Big|\int_{\mathbb{R}^n}
[b(x)-b(z)]\phi_t(y-z)f(z)\,dz\Big|^2\frac{dy\,dt}{t^{n+1}}\Bigg)^{1/2}.$$
Recently, the boundedness of the above commutators associated with intrinsic
square functions on the weighted Lebesgue spaces, weighted Morrey spaces,
weighted Orlicz--Morrey spaces, Musielak--Orlicz Morrey spaces and weighted
Hardy spaces has also been investigated by many authors (see \cite{HW,Wang2012,
GOS,LNYZ}). In particular, Han and Wu \cite{HW} characterized a class of
weighted ${\rm BMO}$-spaces via the boundedness of the above commutators.
More precisely, let $p\in(0,1]$,  and let $w\in A_\infty(\mathbb{R}^n)$  satisfy
$\int_{\mathbb{R}^n}\frac{w(x)}{(1+|x|)^{np}}dx<\infty$. A locally integrable
function $b$ is said to be in $\mathcal{BMO}_{w,p}(\mathbb{R}^n)$ if
$$\|b\|_{\mathcal{BMO}_{w,p}(\mathbb{R}^n)}:=\sup\limits_{B\subset\mathbb{R}^n}\Big(\frac{1}{w(B)}\int_{B^c}\frac{w(x)}{|x-x_B|^{np}}dx\Big)^{1/p}\int_{B}|b(y)-b_B|dy<\infty,$$
where the supremum is taken over all balls $B:=B(x_B,r_B)\subset\mathbb{R}^n$
with $x_B\in\mathbb{R}^n$, $r_B\in(0,\infty)$ and $B^c:=\mathbb{R}^n\setminus B$.
Here and hereafter,
$$w(B):=\int_{B}w(z)dz\ \ \ {\rm and}\ \ \ b_B:=\frac{1}{|B|}\int_B b(z)dz.$$
The main results of \cite{HW} can be formulated as follows.

\medskip

\quad\hspace{-20pt}{\bf Theorem A} {\rm (\cite{HW})} {\it Let $0<\alpha\leq1$,
$n/(n+\alpha)<p\leq1$ and $\lambda>3+2\alpha/n$. Suppose that
$w\in A_{p(1+\alpha/n)}(\mathbb{R}^n)$ and satisfies $\int_{\mathbb{R}^n}
\frac{w(x)}{(1+|x|)^{np}}dx<\infty$. Then for $b\in{\rm BMO}(\mathbb{R}^n)$,
the following statements are equivalent:
\begin{itemize}
\item[(i)] $b\in\mathcal{BMO}_{w,p}(\mathbb{R}^n)$;
\item[(ii)] $g_{\alpha,b}$ is bounded from $H^p_{w}(\mathbb{R}^n)$ to
$L^p_{w}(\mathbb{R}^n)$;
\item[(iii)]$S_{\alpha,b}$ is bounded from $H^p_{w}(\mathbb{R}^n)$ to
$L^p_{w}(\mathbb{R}^n)$;
\item[(iv)] $g_{\alpha,b}^{\lambda,*}$ is bounded from $H^p_{w}(\mathbb{R}^n)$ to
$L^p_{w}(\mathbb{R}^n)$.
\end{itemize}}

\medskip
It is worth noting that $\mathcal{BMO}_{w,p}(\mathbb{R}^{n})$ was
introduced by Liang et al. \cite{Huy,LKY} who pointed out the following
proper inclusion relationship $\mathcal{BMO}_{w,p}(\mathbb{R}^{n})
\subsetneq \rm{BMO}(\mathbb{R}^{n})$. It is well known that for a standard
Calder\'on--Zygmund operator $T$, the commutator $[b,T]$ is not bounded
from $H^{p}(\mathbb{R}^{n})$ to $L^{p}(\mathbb{R}^{n})$ unless $b$ is
constant (see \cite{HST}). In \cite{Huy,LKY}, the authors showed that $[b,T]$
is bounded from the weighted Hardy space $H_{w}^{p}(\mathbb{R}^{n})$
to the weighted Lebesgue space $L_{w}^{p}(\mathbb{R}^{n})$ under the
condition $b\in \mathcal{BMO}_{w,p}(\mathbb{R}^{n})$. The above authors
also proved that the commutators $\{[b,R_j]\}_{j=1}^{n}$ of the classical
Riesz transforms are bounded from $H_{w}^{p}(\mathbb{R}^{n})$ to
$L_{w}^{p}(\mathbb{R}^{n})$ if and only if $b\in\mathcal{BMO}_{w,
p}(\mathbb{R}^{n})$. On the other hand, as an important extension of the
above results, Huy and Ky \cite{HK} studied the mapping properties of
$[b,T]$ with $T$ being a standard Calder\'on--Zygmund operator on
Musielak--Orlicz Hardy spaces $H^{\varphi}(\mathbb{R}^{n})$. More
precisely, Huy and Ky \cite{HK} introduced a suitable subspace
$\mathcal{BMO}_{\varphi}(\mathbb{R}^{n})$ (see section \ref{S2} for the
definition) of $\rm{BMO}(\mathbb{R}^{n})$ and proved that $[b,T]$ is
bounded from $H^{\varphi}(\mathbb{R}^{n})$ to $L^{\varphi}(\mathbb{R}^{n})$
if $b\in\mathcal{BMO}_{\varphi}(\mathbb{R}^{n})$. Moreover, if
$b\in \rm{BMO}(\mathbb{R}^{n})$ and the Riesz-transform commutators
$\{[b,R_j]\}_{j=1}^{n}$ are bounded from $H^{\varphi}(\mathbb{R}^{n})$
to $L^{\varphi}(\mathbb{R}^{n})$, then
$b\in\mathcal{BMO}_{\varphi}(\mathbb{R}^{n})$.

In view of Theorem A and the Musielak--Orlicz commutator theory developed by Huy and Ky \cite{HK}, it is natural to ask whether the above BMO-type characterization for commutators of intrinsic square functions continues to hold in the more general Musielak--Orlicz setting. This leads to the following question.

\begin{question}\label{que1.1}
Can $\mathcal{BMO}_{\varphi}(\mathbb{R}^{n})$  be characterized
by the boundedness of commutators associated with intrinsic square
functions from $H^{\varphi}(\mathbb{R}^{n})$ to $L^{\varphi}(\mathbb{R}^{n})$?
\end{question}

An affirmative answer to this question would extend the weighted characterization in Theorem A to the Musielak--Orlicz framework and, at the same time, complement the results of Huy and Ky \cite{HK} by replacing Calder\'on--Zygmund operators with intrinsic square functions. This constitutes one of the main motivations of the present paper. The passage from weighted Hardy spaces to the Musielak--Orlicz setting is not merely formal. It requires estimates adapted to the growth function $\varphi$ and the corresponding atomic structure of $H^\varphi(\mathbb{R}^n)$. Moreover, establishing a full characterization also requires the converse implication, namely, recovering the ${\rm BMO}_\varphi(\mathbb{R}^n)$ condition on the symbol from the boundedness of the commutators.

Before stating our first main result, let us introduce some notation. Let $\varphi=\varphi(x,t)$ be a growth function on $\mathbb R^n\times[0,\infty)$. We denote by $i(\varphi)$ its critical uniformly lower type index and by $\mathbb A_q$ the uniformly Muckenhoupt classes; see Section \ref{S2} for their precise definitions.

\begin{theorem}\label{thm1.1}
Let $0<\alpha\leq1$, $\lambda>2+2\alpha/n$ and let $\varphi$ be a
growth function satisfying $\varphi\in\mathbb{A}_{i(\varphi)(1+\alpha/n)}$,
$\frac{n}{n+\alpha}<i(\varphi)\leq1$ and $\int_{\mathbb{R}^n}\varphi(x,
(1+|x|)^{-n})dx<\infty$. Then, for $b\in\mathrm{BMO}(\mathbb{R}^n)$, the
following statements are equivalent:
\begin{itemize}
\item[(i)] $b\in\mathcal{BMO}_{\varphi}(\mathbb{R}^n)$;
\item[(ii)] $g_{\alpha,b}$ is bounded from $H^\varphi(\mathbb{R}^n)$
to $L^\varphi(\mathbb{R}^n)$;
\item[(iii)] $S_{\alpha,b}$ is bounded from $H^\varphi(\mathbb{R}^n)$ to
$L^\varphi(\mathbb{R}^n)$;
\item[(iv)] $g_{\alpha,b}^{\lambda,*}$ is bounded from $H^\varphi(\mathbb{R}^n)$
to $L^\varphi(\mathbb{R}^n)$.
\end{itemize}
\end{theorem}

\begin{remark}\label{rem1.1}
Theorem \ref{thm1.1} extends Theorem A to the more general Musielak--Orlicz
Hardy setting. It should be pointed out that the index range
$\lambda>3+2\alpha/n$ in Theorem A is extended to the case $\lambda>2+2\alpha/n$.
\end{remark}

\subsection{Endpoint estimates for commutators of intrinsic square functions}
The endpoint boundedness of commutators of various operators on Hardy spaces
has also been an important topic in harmonic analysis.
P\'erez \cite{Per} first showed that, for a standard Calder\'on--Zygmund operator $T$, the commutator $[b,T]$ is not bounded from $H^1(\mathbb R^n)$ to $L^1(\mathbb R^n)$ unless $b$ is constant (see also \cite{HST}).
  However, P\'erez proposed a
complementary approach to the endpoint problem for various commutators on Hardy
spaces, which is to retain the symbol $b$ in a {\rm BMO}-type class and instead
restrict the domain to a suitable $b$-adapted Hardy space. More precisely, P\'erez
\cite{Per} observed that the above boundedness can be valid if $H^1(\mathbb{R}^n)$
is replaced by an appropriate subspace of $H^1(\mathbb{R}^n)$. The development of
this approach is closely connected with the product structure of Hardy and
${\rm BMO}$ functions. Bonami et al. \cite{BIJZ} showed that the product of an
$H^{1}(\mathbb{R}^{n})$ function and a $\rm{BMO}(\mathbb{R}^{n})$ function can
be decomposed into an integrable part and a distribution belonging to a suitable
Hardy--Orlicz space. Subsequently, Bonami et al. \cite{BGK} obtained a bilinear
version of this decomposition. Based on this bilinear structure, Ky \cite{Ky2013}
found the largest subspace $H_b^{1}(\mathbb{R}^{n})$ of $H^{1}(\mathbb{R}^{n})$
and established the boundedness of commutators of Calder\'on--Zygmund operators
from $H_b^{1}(\mathbb{R}^{n})$ to $L^{1}(\mathbb{R}^{n})$. Recently, Fang and
Liu \cite{FL} extended this bilinear decomposition to the Orlicz--Hardy setting
and introduced a subspace $H_b^{\Phi}(\mathbb{R}^{n})$ of $H^{\Phi}$, see Section \ref{S3} for the definitions of
$H_b^{\Phi}(\mathbb{R}^{n})$ and $H^{\Phi}$. Moreover, the above authors also
established the boundedness of commutators of Calder\'on--Zygmund operators from
$H_b^{\Phi}(\mathbb{R}^{n})$ to $L^{1}(\mathbb{R}^{n})$ and from $H^\Phi(\mathbb R^n)$ to $L^{1,\infty}(\mathbb R^n)$. Motivated by these
results, it is natural to ask the following question:

\begin{question}\label{que1.2}
Let $b\in \mathrm{BMO}_{\Phi}(\mathbb R^n)$. Are the commutators associated with the intrinsic Littlewood--Paley $g$-function, the intrinsic $g_\lambda^*$-function, and the intrinsic Lusin area function bounded from $H_b^\Phi(\mathbb R^n)$ to $L^1(\mathbb R^n)$ and from $H^\Phi(\mathbb R^n)$ to $L^{1,\infty}(\mathbb R^n)$?
\end{question}

The preceeding question motivates the second part of the present paper. Before stating our second main result, we introduce some notation. Let $\Phi=\Phi(t)$ be a growth function independent of the spatial variable, and denote by $i(\Phi)$ and $I(\Phi)$ its critical lower and upper type indices, respectively. Let $H^\Phi(\mathbb R^n)$ be the Orlicz--Hardy space associated with $\Phi$ and $\mathrm{BMO}_\Phi(\mathbb R^n)$ the corresponding BMO-type space.

\begin{theorem}\label{thm1.2}
Let $0<\alpha\leq1$, $2+2\alpha/n<\lambda<\infty$ and let $\Phi$ be a growth
function satisfying $\frac{n}{n+\alpha}<i(\Phi)\leq I(\Phi)<1$ and
$$\begin{array}{ll}
\begin{cases}
\dfrac{1}{I(\Phi)}
>
\dfrac{1}{i(\Phi)}-1,
& n=1,
\\[2mm]
\dfrac{1}{I(\Phi)}
>
\dfrac{
\Big[n(i(\Phi)^{-1}-1)\Big]+n-1}{n},
& n\geq2.
\end{cases}
\end{array}\eqno(1.1)$$
Assume that $b\in\mathrm{BMO}_{\Phi}(\mathbb{R}^n)$. Then the
following conclusions hold:
\begin{itemize}
\item[(i)] $g_{\alpha,b}$ is bounded from $H_b^\Phi(\mathbb{R}^n)$
to $L^1(\mathbb{R}^n)$;
\item[(ii)] $g_{\alpha,b}$ is bounded from $H^\Phi(\mathbb{R}^n)$ to
$L^{1,\infty}(\mathbb{R}^n)$;
\item[(iii)] $S_{\alpha,b}$ is bounded from $H_b^\Phi(\mathbb{R}^n)$
to $L^1(\mathbb{R}^n)$;
\item[(iv)]$S_{\alpha,b}$ is bounded from $H^\Phi(\mathbb{R}^n)$ to
$L^{1,\infty}(\mathbb{R}^n)$;
\item[(v)] $g_{\alpha,b}^{\lambda,*}$ is bounded from
$H_b^\Phi(\mathbb{R}^n)$ to $L^1(\mathbb{R}^n)$;
\item[(vi)] $g_{\alpha,b}^{\lambda,*}$ is bounded from
$H^\Phi(\mathbb{R}^n)$ to $L^{1,\infty}(\mathbb{R}^n)$.
\end{itemize}
\end{theorem}

\begin{remark}\label{1.2}
It should be emphasized that Theorem \ref{thm1.2} and the endpoint
framework developed in Section \ref{S3} can be applied to commutators of
other classical operators, for instance, the Marcinkiewicz integral,
the Lusin area integral, Littlewood--Paley square functions and so on.
\end{remark}

The remainder of this paper is organized as follows. In Section \ref{S2}
we recall some definitions of Musielak--Orlicz Hardy spaces and the
corresponding ${\rm BMO}$ spaces, and later provide a precise proof of
Theorem \ref{thm1.1}. In Section \ref{S3} we present the proof of Theorem
\ref{thm1.2} after recalling the definitions of Orlicz--Hardy spaces and
establishing some preliminary lemmas.

In what follows, the notation $X\lesssim Y$ means $X\le C Y$ for some
constant $C>0$ which is independent of the essential variables depending
on $X$ and $Y$; and $X\simeq Y$ means $X\lesssim Y\lesssim X$. For each
$E\subset\mathbb{R}^n$ we denote $E^c=\mathbb{R}^n\setminus E$. For
$a\in\mathbb{R}$ we denote $[a]=\max\{N\in\mathbb{Z}:N\leq a\}$. Denote
$\mathbb{N}=\{0,1,2,\ldots\}$. For $x=(x_1,\ldots,x_n)\in\mathbb{R}^n$
and $\beta=(\beta_1,\ldots,\beta_n)\in\mathbb{Z}_+^n$, we denote
$x^\beta:=x_1^{\beta_1}\cdots x_n^{\beta_n}$ and $|\beta|:=\beta_1+\cdots
+\beta_n$.

\bigskip

\section{BMO-Type characterizations: Proof of Theorem \ref{thm1.1}}\label{S2}

This section aims to present the proof of Theorem \ref{thm1.1}. We first
recall the basic notions and auxiliary results concerning Musielak--Orlicz
Hardy spaces and the corresponding ${\rm BMO}$ space $\mathcal{BMO}_{\varphi}
(\mathbb{R}^n)$. Secondly, we establish some preliminary lemmas, which are
the main ingredients in the proof of Theorem \ref{thm1.1}. Finally, the proof of
Theorem \ref{thm1.1} will be provided.

\subsection{Musielak--Orlicz Hardy spaces}

Before presenting the definitions of Musielak--Orlicz Hardy spaces, let
us recall some basic notions.

\begin{definition}\label{def2.1}
\begin{itemize}
\item[(i)] {\rm (Musielak--Orlicz function)} A function $\varphi:\mathbb{R}^n
\times[0,\infty)\rightarrow[0,\infty)$ is called a Musielak--Orlicz
function if, for every $x\in\mathbb{R}^n$, $\varphi(x,\cdot)$ is an Orlicz
function on $[0,\infty)$ and, for every $t\in[0,\infty)$, $\varphi(\cdot,t)$
is measurable on $\mathbb{R}^n$.

\item[(ii)] {\rm (Uniformly lower and upper type $p$ functions)} Let $\varphi:
\mathbb{R}^n\times[0,\infty)\rightarrow[0,\infty)$ be a Musielak--Orlicz
function. We say that $\varphi$ is of uniformly lower type $p$ if there
exists a constant $C>0$ such that $\varphi(x,st)\leq Cs^p\varphi(x,t)$
for all $x\in\mathbb{R}^n$, $t\geq0$ and $s\in(0,1]$.
Analogously, $\varphi$ is said to be of uniformly upper type $p$ if the
preceding inequality holds for all $s\in[1,\infty)$.
\item[(iii)] {\rm (Critical uniformly lower and upper type indices)} Let $\varphi:
\mathbb{R}^n\times[0,\infty)\rightarrow[0,\infty)$ be a Musielak--Orlicz
function. The critical uniformly lower and upper type indices are defined,
respectively, by
$$i(\varphi):=\sup\Big\{p\in\mathbb{R}:\varphi\ \text{is of uniformly lower type }p\Big\}$$
and
$$I(\varphi):=\inf\Big\{p\in\mathbb{R}:\varphi\ \text{is of uniformly upper type }p\Big\}.$$

\item[(iv)] {\rm (Uniformly Muckenhoupt condition, \cite{K2014})} Let $q\in[1,\infty)$.
A function $\varphi:\mathbb{R}^n\times[0,\infty)\rightarrow[0,\infty)$ is said to
satisfy the uniformly Muckenhoupt condition $\mathbb{A}_q$, denoted by
$\varphi\in\mathbb{A}_q$, if the following conditions hold uniformly with respect
to $t\in(0,\infty)$, where
$$[\varphi]_{\mathbb{A}_q}:=\sup\limits_{t\in(0,\infty)}\sup\limits_{B\subset\mathbb{R}^n}\Big\{
\frac{1}{|B|}\int_B\varphi(x,t)\,dx\Big(\frac{1}{|B|}\int_B\varphi(x,t)^{-1/(q-1)}\,dx\Big)^{q-1}\Big\}<\infty,\ \ q\in(1,\infty),$$
$$[\varphi]_{\mathbb{A}_1}:=\sup\limits_{t\in(0,\infty)}\sup\limits_{B\subset\mathbb{R}^n}\Big\{\frac{1}{|B|}\int_B\varphi(x,t)dx\operatorname*{ess\,sup}_{x\in B}\varphi(x,t)^{-1}\Big\}<\infty.$$
Denote $\mathbb{A}_{\infty}:=\bigcup_{q\in[1,\infty)}\mathbb{A}_q$ and, for
$\varphi\in\mathbb{A}_{\infty}$, $q(\varphi):=\inf\{q\in[1,\infty):
\varphi\in\mathbb{A}_q\}$.
\item[(v)] {\rm (Growth function)} A function $\varphi:\mathbb{R}^n\times[0,\infty)
\rightarrow[0,\infty)$ is called a growth function if the following conditions
are satisfied: $\varphi$ is a Musielak--Orlicz function;
$\varphi \in \mathbb{A}_{\infty}$; $\varphi$ is of uniformly lower type $p$
for some $p \in(0,1]$ and of uniformly upper type 1.
\end{itemize}
\end{definition}

The following is an important property for an $\mathbb{A}_q$ function.

\begin{lemma}{\rm(\cite{K2014})}\label{lem2.1}
Let $\varphi\in\mathbb{A}_q$ with $q \in[1,\infty)$. Then for any ball
$B\subset\mathbb{R}^n$, $\lambda\in(1,\infty)$ and $t\in(0,\infty)$,
$$\varphi(\lambda B,t)\lesssim\lambda^{n q}\varphi(B,t),$$
where $\varphi(B,t):=\int_B \varphi(x,t)dx$.
\end{lemma}

Now we introduce the Musielak--Orlicz spaces and Musielak--Orlicz Hardy
spaces.

\begin{definition}\label{2.2}
\begin{itemize}
\item[(i)] {\rm (Musielak--Orlicz space $L^{\varphi}(\mathbb{R}^n)$)} The
Musielak--Orlicz space $L^{\varphi}(\mathbb{R}^n)$ is defined as the space
of all Lebesgue measurable functions $f:\mathbb{R}^n
\rightarrow\mathbb{R}$ such that there exists some $\eta\in(0,\infty)$ for
which the following condition holds:
$$\int_{\mathbb{R}^n}\varphi\Big(x,\frac{|f(x)|}{\eta}\Big)dx<\infty,$$
equipped with the (quasi-)norm
$$\|f\|_{L^{\varphi}(\mathbb{R}^n)}:=\inf\Big\{\eta \in(0,\infty):\int_{\mathbb{R}^n}\varphi\Big(x,\frac{|f(x)|}{\eta}\Big)dx \leq 1\Big\}.$$
For any measurable set $A\subset\mathbb{R}^n$, we define
$$\|f\|_{L^{\varphi}(A)}:=\inf\Big\{\eta>0: \int_A \varphi\Big(x,\frac{|f(x)|}{\eta}\Big)dx\leq 1\Big\}.$$
\item[(ii)] {\rm (Musielak--Orlicz Hardy space $H^{\varphi}(\mathbb{R}^n)$,
\cite{K2014})} The Musielak--Orlicz Hardy space $H^{\varphi}(\mathbb{R}^n)$
is the space of all distributions $f\in \mathcal{S}'(\mathbb{R}^n) $ and
$\psi\in \mathcal S(\mathbb{R}^n)$, such that
$\mathcal{M}_\psi f \in L^{\varphi}(\mathbb{R}^n)$ with the (quasi-)norm
$$\|f\|_{H^{\varphi}(\mathbb{R}^n)}:=\|\mathcal{M}_\psi f\|_{L^{\varphi}(\mathbb{R}^n)},$$
where $\mathcal{M}_\psi f(x):=\sup_{t\in(0,\infty)}|f*\psi_t(x)|$ denotes
the maximal function of a distribution $f$ for each $x\in \mathbb{R}^n$
and $\psi_t(\cdot)=t^{-n}\psi(t^{-1}\cdot)$.
\end{itemize}
\end{definition}

It should be pointed out that the Musielak--Orlicz Hardy space
$H^{\varphi}(\mathbb{R}^n)$ was introduced by Ky in \cite{K2014} and
provides a unified framework for several classical Hardy-type spaces
(see also \cite{FMY,YLK}). In fact, when $\varphi(x,t)=t^p$ for
$0<p\leq1$, $H^{\varphi}(\mathbb{R}^n)$ reduces to the classical
Hardy space $H^p(\mathbb{R}^n)$. When $\varphi(x,t)=\omega(x)t^p$
with $\omega\in A_\infty(\mathbb{R}^n)$ and $0<p\leq1$, the space
$H^{\varphi}(\mathbb{R}^n)$ coincides with the weighted Hardy space
$H_\omega^p(\mathbb{R}^n)$.

Before introducing $H^{\varphi}$ atom, let us recall the definition
of $L_\varphi^q(\mathbb{R}^n)$ spaces.

\begin{definition}\label{def2.3}
For each ball $B$ in $\mathbb{R}^n$, we denote
$L_{\varphi}^q(B), q\in(q(\varphi),\infty]$, the set of all measurable
functions $f$ on $\mathbb{R}^n$ satisfying $\supp f \subset B$ such that
$$\|f\|_{L_{\varphi}^q(B)}:=
\begin{cases}\displaystyle\sup\limits_{t>0}\Big(\frac{1}{\varphi(B,t)}{\int_B|f(x)|^q \varphi(x, t) d x}\Big)^{1 / q}<\infty, & 1 \leq q<\infty, \\ \|f\|_{L^{\infty}(\mathbb R^n)}<\infty, & q=\infty.\end{cases}
$$
\end{definition}

Now we introduce the definition of $H^{\varphi}$ atom.

\begin{definition}\label{def2.4}
Let $\varphi$ be a growth function and set
$m(\varphi):=[n(\frac{q(\varphi)}{i(\varphi)}-1)]$. Let $q\in(q(\varphi),
\infty]$ and $s\in\mathbb{N}$ satisfy $s\geq m(\varphi)$. A measurable
function $a$ is called an $(H^\varphi,q,s)$-atom associated with a ball
$B\subset\mathbb{R}^n$ if
\begin{itemize}
\item[(i)] ${\rm supp}a\subset B$;
\item[(ii)] $\|a\|_{L_\varphi^q(B)}\leq\|\chi_B\|_{L^\varphi(\mathbb{R}^n)}^{-1}$;
\item[(iii)] $\displaystyle\int_{\mathbb{R}^n}a(x)x^\beta\,dx=0$ for every
multi-index $\beta\in\mathbb{N}^n$ with $|\beta|\leq s$.
\end{itemize}
\end{definition}

The following atomic extension criterion is very useful in the proof of
Theorem \ref{thm1.1}.

\begin{lemma}{\rm(\cite{K2014})}\label{lem2.2}
Let $\varphi$ be a growth function and $s\in\mathbb{Z}_+$ satisfy
$s\geq m(\varphi)$. Let $\mathcal{X}$ be a quasi-Banach space.
Suppose that one of the following conditions holds:

\begin{itemize}
\item[(i)] $q\in(q(\varphi),\infty)$, and $T$ is defined on the space
of finite linear combinations of $(H^\varphi,q,s)$-atoms and satisfies
$$\sup\{\|Ta\|_{\mathcal{X}}:a\ \text{is an }(H^\varphi,q,s)\text{-atom}\}<\infty;$$
\item[(ii)] $T$ is defined on the space of finite linear combinations
of continuous $(H^\varphi,\infty,s)$-atoms and satisfies
$$\sup\{\|Ta\|_{\mathcal{X}}:a\ \text{is a continuous }(H^\varphi,\infty,s)\text{-atom}\}<\infty.$$
\end{itemize}
Then $T$ admits a bounded extension from $H^\varphi(\mathbb{R}^n)$
into $\mathcal{X}$.
\end{lemma}

\begin{remark}\label{rem2.1}
In view of the assumptions of Theorem \ref{thm1.1}, we have
$\varphi\in\mathbb{A}_{i(\varphi)(1+\alpha/n)}$ and $i(\varphi)
>\frac{n}{n+\alpha}$. By the openness property of the Muckenhoupt classes,
$q(\varphi)<i(\varphi)(1+\frac{\alpha}{n})$. Hence,
$$n\Big(\frac{q(\varphi)}{i(\varphi)}-1\Big)<\alpha\leq1,$$
and therefore $m(\varphi)=0$.
\end{remark}

\subsection{$\mathcal{BMO}_{\varphi}(\mathbb{R}^n)$ spaces}

In this subsection, we give the definitions of $\mathcal {BMO}_{\varphi}
(\mathbb{R}^n)$ space, $\mathrm {BMO}(\mathbb{R}^n)$ space and some basic
facts.

\begin{definition}\label{def2.5}
\begin{itemize}
\item[(i)] {\rm ($\mathcal{BMO}_{\varphi}(\mathbb{R}^n)$ spaces)} Let $\varphi$
be a growth function with
$$\int_{\mathbb{R}^n}\varphi(x,(1+|x|)^{-n})dx<\infty.$$
A locally integrable function $b\in L_{\rm loc}^1(\mathbb{R}^n)$ is said to be
in $\mathcal{BMO}_{\varphi}(\mathbb{R}^n)$ if
$$\|b\|_{\mathcal{BMO}_{\varphi}(\mathbb{R}^n)}:=\sup\limits_{B \subset \mathbb{R}^n}\Big\{\frac{1}{\|\chi_B\|_{L^{\varphi}(\mathbb{R}^n)}}
\int_B|b(x)-b_B| d x\| | \cdot-x_B|^{-n} \|_{L^{\varphi}(B^c)}\Big\}<\infty,$$
where the supremum is taken over all balls $B\subset \mathbb{R}^n$.
\item[(ii)] {\rm ($\mathrm{BMO}(\mathbb{R}^n)$ space)} The
 space $\rm {BMO}(\mathbb R^n)$ consists of all locally integrable functions
$b:\mathbb{R}^n\rightarrow\mathbb{R}$ satisfying
$$\|b\|_{\mathrm{BMO}(\mathbb{R}^n)}:=\sup\limits_{B\subset\mathbb{R}^n}\frac{1}{|B|}\int_B|b(x)-b_B|dx<\infty.$$
\end{itemize}
\end{definition}

\begin{remark}\label{rem2.2}
\begin{itemize}
\item[(i)] When $\varphi(x,t)=w(x)t^p$ with $w\in A_\infty(\mathbb{R}^n)$
and $0<p\leq1$, the space $\mathcal{BMO}_{\varphi}(\mathbb{R}^n)$ reduces to
$\mathcal{BMO}_{w,p}(\mathbb{R}^n)$.
\item[(ii)] The space $\mathcal{BMO}_{\varphi}(\mathbb{R}^n)$ is  a nontrivial
  subspace of $\mathrm{BMO}(\mathbb{R}^n)$. Moreover,
$$\|b\|_{\mathrm{BMO}(\mathbb{R}^n)} \lesssim \|b\|_{\mathcal{B M O}_{\varphi}(\mathbb{R}^n)},\ \ \forall b \in \mathcal{BMO}_{\varphi}(\mathbb{R}^n).$$
\end{itemize}
\end{remark}

The following weighted John--Nirenberg type estimate will be useful in our proof.

\begin{lemma}[{\cite{LKY}}]\label{lem2.3}
Assume that $\varphi\in\mathbb{A}_\infty$ and $q\in(0,\infty)$. Then,
uniformly for $t>0$ and every ball $B\subset\mathbb{R}^n$,
$$\Big(\frac{1}{\varphi(B,t)}\int_B |b(x)-b_B|^q\varphi(x,t)dx\Big)^{1/q}
\lesssim\|b\|_{\mathrm{BMO}(\mathbb{R}^n)},\ \ \ \forall b\in\mathrm{BMO}(\mathbb{R}^n).$$
\end{lemma}

\begin{lemma}{\rm(\cite{HK})}\label{lem2.4}
Let $\varphi$ be a growth function and $q\in(q(\varphi),\infty]$. Then
$$\int_B\varphi(x,|f(x)|)dx\lesssim\varphi(B,\|f\|_{L_\varphi^q(B)})$$
and
$$\int_B\varphi(x,|b(x)-b_B|)dx\lesssim\varphi(B,\|b\|_{\mathrm{BMO}(\mathbb{R}^n)})$$
for all balls $B\subset\mathbb{R}^n$, $f\in L_\varphi^q(B)$ and
$b\in\mathrm{BMO}(\mathbb{R}^n)$.
\end{lemma}

\begin{lemma}[{\cite{HK}}]\label{lem2.5}
Let $\varphi$ be a growth function satisfying $\int_{\mathbb{R}^n}\varphi
(x,(1+|x|)^{-n})dx<\infty$. Then, for every $b\in\mathcal{BMO}_{\varphi}
(\mathbb{R}^n)$ and every $(H^\varphi,\infty,0)$-atom $a$ associated with
a ball $B\subset\mathbb{R}^n$,
$$\|(b-b_B)a\|_{H^\varphi(\mathbb{R}^n)}\lesssim\|b\|_{\mathcal{BMO}_{\varphi}(\mathbb{R}^n)}.$$
\end{lemma}

We shall also use the following elementary estimate in the converse argument.

\begin{lemma}[{\cite{LKY}}]\label{lem2.6}
Suppose that $f$ is a measurable function with $\operatorname{supp}f\subset
B:=B(x_B,r_B)$. Then, for any $x\in\mathbb{R}^n\setminus B$,
$$\frac{1}{|x-x_B|^n}\Big|\int_B f(y)dy\Big|\lesssim\mathcal{M}_{\psi}f(x).$$
\end{lemma}

\medskip

\subsection{Proof of Theorem \ref{thm1.1}}

We first establish two auxiliary lemmas and then prove Theorem \ref{thm1.1}; by Remark \ref{rem2.1}, it suffices to use atoms with $s=0$.

\begin{lemma}\label{lem2.7}
Let $0<\alpha\leq1$, and let $\varphi$ be a growth function satisfying
$\varphi\in\mathbb{A}_{i(\varphi)(1+\alpha/n)}$ and
$n/(n+\alpha)<i(\varphi)\leq1$. Then, for every
$b\in\mathrm{BMO}(\mathbb{R}^n)$ and every $(H^\varphi,\infty,0)$-atom
$a$ associated with a ball $B=B(x_0,r)$,
$$\big\||b-b_B|g_\alpha(a)\big\|_{L^\varphi(\mathbb{R}^n)}\lesssim \|b\|_{\mathrm{BMO}(\mathbb{R}^n)}.$$
\end{lemma}
\begin{proof}
Let $q_0\in(q(\varphi),\infty)$ and $q_1,q_2\in(1,\infty)$ such that
${1}/{q_0}={1}/{q_1}+{1}/{q_2}$. Then $\varphi(\cdot,t)\in A_{q_2}$
uniformly in $t>0$. By H\"older's inequality, the weighted
$L^{q_2}$-boundedness of $g_\alpha$ (see \cite{W}), and
Lemmas \ref{lem2.1} and \ref{lem2.3},
$$\begin{array}{ll}
&\displaystyle\Big(\frac{1}{\varphi(2B,t)}\int_{2B}|(b(x)-b_{2B})g_\alpha(a)(x)|^{q_0}\varphi(x,t)dx\Big)^{1/q_0}\\
&\leq\displaystyle\Big(\frac{1}{\varphi(2B,t)}\int_{2B}|b(x)-b_{2B}|^{q_1}\varphi(x,t)dx\Big)^{1/q_1}\Big(\frac{1}{\varphi(2B,t)}
\int_{2B}|g_\alpha(a)(x)|^{q_2}\varphi(x,t)dx\Big)^{1/q_2} \\
&\lesssim\displaystyle\|b\|_{\mathrm{BMO}(\mathbb{R}^n)}\Big(\frac{1}{\varphi(2B,t)}\int_{\mathbb{R}^n}|a(x)|^{q_2}\varphi(x,t)dx\Big)^{1/q_2}\\
&\lesssim\|b\|_{\mathrm{BMO}(\mathbb{R}^n)}\|\chi_B\|_{L^{\varphi}(\mathbb{R}^n)}^{-1},
\end{array}$$
for every $t>0$. Hence, we get by Lemmas \ref{lem2.1} and \ref{lem2.4} that
$$\begin{array}{ll}
&\|(b-b_{2B})g_\alpha(a)\|_{L^{\varphi}(2B)}\lesssim\displaystyle\inf\Big\{\eta>0:\varphi\Big(2B,\frac{\|(b(\cdot)-b_{2B})g_\alpha(a)(\cdot)\|_{L_{\varphi}^{q_0}(2B)}}{\eta}\Big)\leq 1\Big\}\\
&\qquad\qquad\qquad\qquad\qquad\lesssim\displaystyle\inf\Big\{\eta>0:\varphi(B, \frac{\|b\|_{\mathrm{BMO}(\mathbb{R}^n)}\|\chi_B\|_{L^{\varphi}(\mathbb{R}^n)}^{-1}}{\eta}\Big)\leq 1\Big\}\\
&\qquad\qquad\qquad\qquad\qquad\lesssim\|b\|_{\mathrm {BMO}(\mathbb{R}^n)}.
\end{array}$$
In view of $|b_{2B}-b_B|\lesssim\|b\|_{\mathrm{BMO}(\mathbb{R}^n)}$
and the boundedness of $g_\alpha:H^\varphi(\mathbb{R}^n)\rightarrow
L^\varphi(\mathbb{R}^n)$, we further obtain
$$\|(b-b_B)g_\alpha(a)\|_{L^\varphi(2B)}\lesssim \|b\|_{\mathrm{BMO}(\mathbb{R}^n)}.$$

For $x\in(2B)^c$ and $\phi\in\mathcal C_\alpha$, the cancellation
condition of $a$ gives
\begin{align*}
|a*\phi_t(x)|
&=\left|\int_B[\phi_t(x-y)-\phi_t(x-x_0)]a(y)\,dy\right|\\
&\lesssim \frac{r^\alpha}{t^{n+\alpha}}\int_B|a(y)|\,dy
\lesssim \frac{r^{n+\alpha}}{t^{n+\alpha}}
\|\chi_B\|_{L^\varphi(\mathbb{R}^n)}^{-1}.
\end{align*}
Observe that if $a*\phi_t(x)\neq0$, then there exists $y\in B$ such that
$|x-y|<t$. Moreover, for $x\in(2B)^c$,
$$t>|x-y|\geq |x-x_0|-|y-x_0|\geq\frac{|x-x_0|}{2}.$$
It follows that
$$\begin{array}{ll}
&|g_\alpha(a)(x)|^2\lesssim\displaystyle\|\chi_B\|_{L^\varphi(\mathbb{R}^n)}^{-2}r^{2\alpha+2n}\int_{|x-x_0|/2}^\infty\frac{dt}{t^{2(n+\alpha)+1}}\\
&\qquad\qquad\quad\lesssim\displaystyle\|\chi_B\|_{L^\varphi(\mathbb{R}^n)}^{-2}r^{2\alpha+2n}\frac{1}{|x-x_0|^{2n+2\alpha}}.
\end{array}$$
In particular, for $x\in2^{j+1}B\setminus2^jB$, $j\geq1$,
$$g_\alpha(a)(x)\lesssim2^{-j(n+\alpha)}\|\chi_B\|_{L^\varphi(\mathbb{R}^n)}^{-1}.$$

By the openness property of the uniformly Muckenhoupt classes, we may
choose $\mathfrak p\in(0,i(\varphi))$ and $\mathfrak q\in(1,
\mathfrak p(1+\alpha/n))$ such that $\varphi\in\mathbb A_{\mathfrak q}$.
Moreover, by the uniformly upper type $1$ property of $\varphi$, Lemmas
\ref{lem2.1} and~\ref{lem2.4}, and the standard estimate
$|b_{2^jB}-b_B|\lesssim j\|b\|_{\mathrm{BMO}(\mathbb{R}^n)}$, we have
$$\int_{2^jB}\varphi\Big(x,\frac{|b(x)-b_B|}{\eta}\Big)dx\lesssim2^{jn\mathfrak q}(j+1)
\varphi\Big(B,\frac{\|b\|_{\mathrm{BMO}(\mathbb{R}^n)}}{\eta}\Big)$$
for every $\eta>0$. Using the uniformly lower type $\mathfrak p$
property of $\varphi$ and the preceding pointwise estimate, we obtain
$$\begin{array}{ll}
&\displaystyle\int_{(2B)^c}\varphi\Big(x,\frac{|b(x)-b_B|g_\alpha(a)(x)}{\eta\|\chi_B\|_{L^\varphi(\mathbb{R}^n)}^{-1}}\Big)dx\\
&\lesssim\displaystyle\sum\limits_{j=1}^\infty2^{-j(\mathfrak p(n+\alpha)-n\mathfrak q)}(j+2)\varphi\Big(B,\frac{\|b\|_{\mathrm{BMO}(\mathbb{R}^n)}}{\eta}\Big)
\lesssim\varphi\Big(B,\frac{\|b\|_{\mathrm{BMO}(\mathbb{R}^n)}}{\eta}\Big),
\end{array}$$
where  the second inequality follows from the fact that $\mathfrak q<\mathfrak p(1+\alpha/n)$. In view of the
definition of the Luxemburg norm, one gets
$$\|(b-b_B)g_\alpha(a)\|_{L^\varphi((2B)^c)}\lesssim \|b\|_{\mathrm{BMO}(\mathbb{R}^n)}.$$
Combining the local and far-field estimates completes the proof of Lemma
\ref{lem2.7}.
\end{proof}

\begin{lemma}\label{lem2.8}
Let $0<\alpha\leq1$ and $2+2\alpha/n<\lambda<\infty$, and let $\varphi$ be a
growth function satisfying $\varphi\in\mathbb{A}_{i(\varphi)(1+\alpha/n)}$ and
$n/(n+\alpha)<i(\varphi)\leq1$. Then, for every $b\in\mathrm{BMO}(\mathbb{R}^n)$
and every $(H^\varphi,\infty,0)$-atom $a$ associated with a ball $B=B(x_0,r)
\subset\mathbb{R}^n$,
$$\big\||b-b_B|g_{\lambda,\alpha}^{*}(a)\big\|_{L^\varphi(\mathbb{R}^n)}
\lesssim\|b\|_{\mathrm{BMO}(\mathbb{R}^n)}.$$
\end{lemma}
\begin{proof}
As in the proof of Lemma \ref{lem2.7}, we decompose $\mathbb{R}^n$ into
$8B$ and $(8B)^c$. The estimate over $8B$ follows from the weighted
boundedness of $g_{\lambda,\alpha}^{*}$ (see \cite{W}) and the
same argument used for the local part of Lemma \ref{lem2.7}. Thus, it
remains to consider the far-field part.

Let $x\in 2^{j+4}B\setminus 2^{j+3}B$, $j\in\mathbb{N}_0$. We first establish
$$g_{\lambda,\alpha}^{*}(a)(x)\lesssim 2^{-j(n+\alpha)}\|\chi_B\|_{L^\varphi(\mathbb{R}^n)}^{-1}.\eqno(2.1)$$
For any $\phi\in\mathcal{C}_\alpha$, by the cancellation condition of $a$,
$$\begin{array}{ll}
&\displaystyle\Big|\int_{\mathbb{R}^n}\phi_t(y-z)a(z)dz\Big|=\displaystyle\Big|\int_B\big[\phi_t(y-z)-\phi_t(y-x_0)\big]a(z)dz\Big|\\
&\qquad\qquad\qquad\qquad\qquad\lesssim\displaystyle\int_B\frac{|z-x_0|^\alpha}{t^{n+\alpha}}|a(z)|dz
\lesssim \frac{r^{n+\alpha}}{t^{n+\alpha}}\|\chi_B\|_{L^\varphi(\mathbb{R}^n)}^{-1}.
\end{array}$$
Moreover, since $\operatorname{supp}\phi\subset B(0,1)$ and
$\|\phi\|_{L^\infty}\lesssim1$,  we have
$$A_\alpha(a)(y,t)\lesssim \|\chi_B\|_{L^\varphi(\mathbb{R}^n)}^{-1}
\min\left\{1,\left(\frac{r}{t}\right)^{n+\alpha}\right\}.$$
Observe that $\operatorname{supp}a\subset B(x_0,r)$ and $\operatorname{supp}
\phi_t\subset B(0,t)$. It follows that $A_\alpha(a)(y,t)=0$ whenever
$|y-x_0|>r+t$. Therefore,
$$\begin{array}{ll}
&|g_{\lambda,\alpha}^{*}(a)(x)|^2\lesssim\displaystyle\|\chi_B\|_{L^\varphi(\mathbb{R}^n)}^{-2}\iint_{|y-x_0|<r+t}
\Big(\frac{t}{t+|x-y|}\Big)^{\lambda n}\min\Big\{1,\Big(\frac{r}{t}\Big)^{2(n+\alpha)}\Big\}\frac{dydt}{t^{n+1}}\\
&\qquad\qquad\qquad=:\|\chi_B\|_{L^\varphi(\mathbb{R}^n)}^{-2}(I_1+I_2+I_3),
\end{array}\eqno(2.2)$$
where
$$I_1:=\iint_{|y-x_0|<r+t,\atop 0<t<r}
\Big(\frac{t}{t+|x-y|}\Big)^{\lambda n}\min\Big\{1,\Big(\frac{r}{t}\Big)^{2(n+\alpha)}\Big\}\frac{dydt}{t^{n+1}},$$
$$I_2:=\iint_{|y-x_0|<r+t,\atop r\leq t<|x-x_0|/4}
\Big(\frac{t}{t+|x-y|}\Big)^{\lambda n}\min\Big\{1,\Big(\frac{r}{t}\Big)^{2(n+\alpha)}\Big\}\frac{dydt}{t^{n+1}},$$
$$I_3:=\iint_{|y-x_0|<r+t,\atop t\geq|x-x_0|/4}
\Big(\frac{t}{t+|x-y|}\Big)^{\lambda n}\min\Big\{1,\Big(\frac{r}{t}\Big)^{2(n+\alpha)}\Big\}\frac{dydt}{t^{n+1}}.$$

For $I_1$, since $|y-x_0|<r+t<2r$ and $|x-y|\gtrsim |x-x_0|$, we obtain
$$I_1\lesssim \int_0^r\int_{|y-x_0|<2r}\Big(\frac{t}{|x-x_0|}\Big)^{\lambda n}\frac{dy\,dt}{t^{n+1}}
\lesssim\Big(\frac{r}{|x-x_0|}\Big)^{\lambda n}\lesssim\Big(\frac{r}{|x-x_0|}\Big)^{2(n+\alpha)},$$
where the last inequality follows from $\lambda>2+2\alpha/n$.

For $I_2$, in view of $|y-x_0|<r+t<2t$ and $|x-y|\gtrsim |x-x_0|$, one gets
$$\begin{array}{ll}
&I_2 \lesssim \displaystyle\int_r^{|x-x_0|/4}\int_{|y-x_0|<2t}\Big(\frac{t}{|x-x_0|}\Big)^{\lambda n}\Big(\frac{r}{t}\Big)^{2(n+\alpha)}\frac{dydt}{t^{n+1}}\\
&\quad\lesssim\displaystyle\frac{r^{2(n+\alpha)}}{|x-x_0|^{\lambda n}}\int_r^{|x-x_0|/4}t^{\lambda n-2(n+\alpha)-1}dt\lesssim\Big(\frac{r}{|x-x_0|}\Big)^{2(n+\alpha)}.
\end{array}$$
We also note that
$$I_3\lesssim \int_{|x-x_0|/4}^{\infty}t^n\Big(\frac{r}{t}\Big)^{2(n+\alpha)}\frac{dt}{t^{n+1}}
\lesssim\Big(\frac{r}{|x-x_0|}\Big)^{2(n+\alpha)}$$
by
$$\int_{\mathbb{R}^n}\Big(\frac{t}{t+|x-y|}\Big)^{\lambda n}dy\lesssim t^n.$$
The above estimates together with (2.2) imply that
$$g_{\lambda,\alpha}^{*}(a)(x)\lesssim \|\chi_B\|_{L^\varphi(\mathbb{R}^n)}^{-1}\Big(\frac{r}{|x-x_0|}\Big)^{n+\alpha}
\lesssim 2^{-j(n+\alpha)}\|\chi_B\|_{L^\varphi(\mathbb{R}^n)}^{-1}.$$
This yields (2.1).

Next we choose $\mathfrak{p}\in(0,i(\varphi))$ and $\mathfrak{q}\in(1,
\mathfrak{p}(1+\alpha/n))$ such that $\varphi\in\mathbb{A}_{\mathfrak{q}}$.
Using the uniformly lower type $\mathfrak{p}$ property of $\varphi$,
together with the preceding pointwise estimate, we obtain
$$\begin{array}{ll}
&\displaystyle\int_{(8B)^c}\varphi\Big(x,\frac{|b(x)-b_B|\,g_{\lambda,\alpha}^{*}(a)(x)}
{\eta\|\chi_B\|_{L^\varphi(\mathbb{R}^n)}^{-1}}\Big)dx\\
&\lesssim\displaystyle\sum\limits_{j=0}^{\infty}2^{-j[\mathfrak{p}(n+\alpha)-n\mathfrak{q}]}(j+2)
\varphi\Big(B,\frac{\|b\|_{\mathrm{BMO}(\mathbb{R}^n)}}{\eta}\Big)
\lesssim\displaystyle\varphi\Big(B,\frac{\|b\|_{\mathrm{BMO}(\mathbb{R}^n)}}{\eta}\Big),
\end{array}$$
where the second inequality follows from $\mathfrak{q}<\mathfrak{p}(1+\alpha/n)$. Hence,
$$\begin{array}{ll}
&\|(b-b_B)g_{\lambda,\alpha}^{*}(a)\|_{L^\varphi((8B)^c)}
\lesssim\displaystyle\|\chi_B\|_{L^\varphi(\mathbb{R}^n)}^{-1}\inf\Big\{\eta>0:
\varphi\Big(B,\frac{\|b\|_{\mathrm{BMO}(\mathbb{R}^n)}}{\eta}\Big)\leq1\Big\}\\
&\qquad\qquad\qquad\qquad\qquad\qquad\lesssim \|b\|_{\mathrm{BMO}(\mathbb{R}^n)}.
\end{array}$$
Combining this with (2.1) and the estimate over $8B$, we conclude that
$$\|(b-b_B)g_{\lambda,\alpha}^{*}(a)\|_{L^\varphi(\mathbb{R}^n)}
\lesssim \|b\|_{\mathrm{BMO}(\mathbb{R}^n)}.$$
This completes the proof.
\end{proof}

Before giving the proof of Theorem \ref{thm1.1}, we first recall the
Littlewood--Paley $g$-function and Littlewood--Paley $g_\lambda^*$-function.
Let $\vartheta\in\mathcal{S}(\mathbb{R}^n)$ be a radial function satisfying
$\operatorname{supp}\vartheta\subset\{x\in\mathbb{R}^n:|x|\leq1\}$,
$\int_{\mathbb{R}^n}\vartheta(x)\,dx=0$ and, for every
$\xi\in\mathbb{R}^n\setminus\{0\}$,
$$\int_0^\infty|\widehat{\vartheta}(t\xi)|^2\frac{dt}{t}=1.$$
Let $\vartheta_t(\cdot)=t^{-n}\vartheta(t^{-1}\cdot)$. For
$f\in\mathcal{S}'(\mathbb{R}^n)$, the usual Littlewood--Paley
$g$-function is defined by
$$g(f)(x):=\Big(\int_0^\infty |f*\vartheta_t(x)|^2\frac{dt}{t}\Big)^{1/2},\qquad x\in\mathbb{R}^n.$$
Meanwhile, the Littlewood--Paley $g_\lambda^*$-function is defined by
$$g_\lambda^*(f)(x):=\Big(\int_0^\infty\int_{\mathbb{R}^n}\Big(\frac{t}{t+|x-y|}\Big)^{\lambda n}|f*\vartheta_t(y)|^2\frac{dy\,dt}{t^{n+1}}\Big)^{1/2},\qquad x\in\mathbb{R}^n.$$

Finally, we present the proof of Theorem \ref{thm1.1}.

\begin{proof}[\bf Proof of Theorem \ref{thm1.1}]
{\it Proof of $((i)\Rightarrow(ii))$}. By Lemma \ref{lem2.2}, it suffices
to prove that
$$\|g_{\alpha,b}(a)\|_{L^{\varphi}(\mathbb{R}^n)}\lesssim\|b\|_{\mathcal{BMO}_{\varphi}(\mathbb{R}^n)},\eqno(2.3)$$
for all continuous $(H^{\varphi},\infty,0)$-atom $a$ related to the ball
$B \subset \mathbb{R}^n$. Indeed, it follows from the boundedness of
$g_\alpha$ from $H^{\varphi}(\mathbb{R}^n)$ to $L^{\varphi}(\mathbb{R}^n)$
(see \cite[Theorem~1.6]{LY}) and Lemmas \ref{lem2.5} and \ref{lem2.7} that
$$\begin{array}{ll}
&\|g_{\alpha,b}(a)\|_{L^{\varphi}(\mathbb{R}^n)}\lesssim\|(b-b_B) g_\alpha(a)\|_{L^{\varphi}(\mathbb{R}^n)}+\|g_\alpha((b-b_B)a)\|_{L^{\varphi}(\mathbb{R}^n)}\\
&\qquad\qquad\qquad\quad\lesssim\|b\|_{\mathrm {BMO}(\mathbb{R}^n)}+\|g_\alpha\|_{H^{\varphi}(\mathbb{R}^n)\rightarrow L^{\varphi}(\mathbb{R}^n)}\|(b-b_B) a\|_{H^{\varphi}(\mathbb{R}^n)} \\
&\qquad\qquad\qquad\quad\lesssim\|b\|_{\mathrm {BMO}(\mathbb{R}^n)}+\|b\|_{\mathcal{B M O}_{\varphi}(\mathbb{R}^n)} \lesssim\|b\|_{\mathcal{B M O}_{\varphi}(\mathbb{R}^n)} .
\end{array}$$
This proves (2.3).

{\it Proof of $((ii)\Rightarrow(i))$}. Let $a$ be an $(H^\varphi,\infty,
0)$-atom associated with a ball $B:=B(x_0,r)$. By Lemma \ref{lem2.7},
$$\|g_\alpha((b-b_B)a)\|_{L^\varphi(\mathbb{R}^n)}
\lesssim\|g_{\alpha,b}\|_{H^\varphi(\mathbb{R}^n)\to L^\varphi(\mathbb{R}^n)}
+\|b\|_{\mathrm{BMO}(\mathbb{R}^n)}.$$
Since the classical Littlewood--Paley $g$-function satisfies
$g(f)\lesssim g_\alpha(f)$ (see the proof of \cite[Theorem~1.6]{LY}), it
follows that
$$\|g((b-b_B)a)\|_{L^\varphi(\mathbb{R}^n)}\lesssim
\|g_{\alpha,b}\|_{H^\varphi(\mathbb{R}^n)\to L^\varphi(\mathbb{R}^n)}+\|b\|_{\mathrm{BMO}(\mathbb{R}^n)}.$$
In order to apply the Littlewood--Paley $g$-function characterization of
$H^\varphi(\mathbb{R}^n)$ from \cite[Theorem~4.4]{LHY}, we need to verify
the required behavior at infinity. Set $h:=(b-b_B)a$. Since
$\operatorname{supp}h\subset B$ and $h\in L^1(\mathbb{R}^n)$,  for
every $\phi\in\mathcal S(\mathbb{R}^n)$,
$$\|h*\phi_t\|_{L^\infty(\mathbb{R}^n)}\leq t^{-n}\|h\|_{L^1(\mathbb{R}^n)}\|\phi\|_{L^\infty(\mathbb{R}^n)}\longrightarrow0\qquad\text{as }t\to\infty.$$
Thus $h$ vanishes weakly at infinity. Hence, applying \cite[Theorem~4.4]{LHY},
one gets $(b-b_B)a\in H^\varphi(\mathbb{R}^n)$ and
$$\|(b-b_B)a\|_{H^\varphi(\mathbb{R}^n)}\lesssim\|g_{\alpha,b}\|_{H^\varphi(\mathbb{R}^n)\to L^\varphi(\mathbb{R}^n)}+\|b\|_{\mathrm{BMO}(\mathbb{R}^n)}.$$

For any ball $B:=B(x_0,r)\subset\mathbb{R}^n$ with some $x_0\in \mathbb{R}^n$
and $r\in (0,\infty)$, let
$$\widetilde a:=\frac{1}{2\|\chi_B\|_{L^\varphi(\mathbb R^n)}}(f-f_B)\chi_B,$$
where $f:={\rm sign}(b-b_B)$. It is easy to see that $\tilde a $ is an
$(H^\varphi,\infty,0)$-atom related to the ball $B$. Moreover, for every
$x\notin B$, we get by Lemma \ref{lem2.6} that
$$\begin{array}{ll}
&\displaystyle|x-x_B|^{-n}\frac{1}{2\|\chi_B\|_{L^{\varphi}(\mathbb{R}^n)}}\int_B|b(y)-b_B|dy\\
&=\displaystyle|x-x_B|^{-n}\int_B[b(y)-b_B]\tilde a(y)dy\lesssim \mathcal{M}_{\psi}([b-b_B]\tilde a)(x).
\end{array}$$
Thus, we have
$$\| b\|_{\mathcal{BMO}_{\varphi}{(\mathbb{R}^n)}}\lesssim \big\|g_{\alpha,b}\big\|_{H^\varphi(\mathbb{R}^n)\to L^\varphi(\mathbb{R}^n)}+\|b\|_{\rm BMO(\mathbb{R}^n)}.$$
This completes the proof of $((ii)\Rightarrow(i))$.

{\it Proof of $((i)\Leftrightarrow(iii))$}. By the pointwise comparability
between the intrinsic Lusin area function and the intrinsic Littlewood--Paley
$g$-function (see \cite{GOS,W}), the conclusion for $((i)\Leftrightarrow(iii))$
follows directly from $((i)\Leftrightarrow(ii))$.

{\it Proof of $((i)\Rightarrow(iv))$}. This part follows from Lemmas \ref{lem2.2},
\ref{lem2.5}, and \ref{lem2.8} and the boundedness of $g_{\lambda,\alpha}^{*}:
H^\varphi(\mathbb{R}^n)\rightarrow L^\varphi(\mathbb{R}^n)$ (see
\cite[Theorem~1.8]{LY}).

{\it Proof of $((iv)\Rightarrow(i))$}. By the fact that $g_\lambda^*(f)\lesssim
g_{\lambda,\alpha}^{*}(f)$, the Littlewood--Paley $g_\lambda^*$-function
characterization of $H^\varphi(\mathbb{R}^n)$ from \cite[Theorem~4.8]{LHY},
and the fact that $(b-b_B)a$ vanishes weakly at infinity, we obtain
$$\|(b-b_B)a\|_{H^\varphi(\mathbb{R}^n)}\lesssim
\|g_{\lambda,\alpha,b}^{*}\|_{H^\varphi(\mathbb{R}^n)\to L^\varphi(\mathbb{R}^n)}+\|b\|_{\mathrm{BMO}(\mathbb{R}^n)}.$$
The rest of the proof is similar to that of $((ii)\Rightarrow(i))$. We omit the
details.
\end{proof}

\bigskip

\section{Endpoint estimates: Proof of Theorem \ref{thm1.2}}\label{S3}

In this section we shall prove some endpoint estimates for commutators
of intrinsic square functions including the intrinsic Littlewood--Paley
$g$-function, the intrinsic $g_{\lambda}^{*}$-function and the intrinsic
Lusin area function on Orlicz--Hardy spaces. In contrast to the
Musielak--Orlicz framework considered in Section \ref{S2}, the growth
function we consider  here is independent of the spatial variable.
First, we introduce some notation and preliminary lemmas in
Subsection 3.1. The proof of Theorem \ref{thm1.2} will be given in
Subsection 3.2.

\subsection{Orlicz--Hardy spaces}

Before presenting the definitions of Orlicz--Hardy spaces, let
us recall some basic notions.

\begin{definition}\label{def3.1}
\begin{itemize}
\item[(i)] {\rm (Orlicz function)} A nondecreasing function
$\Phi:[0,\infty)\to[0,\infty)$ is called an Orlicz function if
$\Phi(0)=0$, $\Phi(t)>0$ for all $t>0$, and $\lim_{t\to\infty}\Phi(t)=\infty$.
\item[(ii)] {\rm (Critical lower and upper type indices)} Let $\Phi:[0,\infty)
\to[0,\infty)$ be an Orlicz function. For $q_1\geq0$, we say that $\Phi$ is of
lower type $q_1$ if there exists a constant $C_{q_1}>0$ such that
$\Phi(st)\leq C_{q_1}s^{q_1}\Phi(t)$ for all $t\geq0$ and $s\in(0,1]$. Similarly,
$\Phi$ is said to be of upper type $q_2$, with $q_2\geq0$, if there exists a
constant $C_{q_2}>0$ such that $\Phi(st)\leq C_{q_2}s^{q_2}\Phi(t)$ for all
$t\geq0$ and $s\in[1,\infty)$. The corresponding critical lower and upper type
indices are defined, respectively, by
$$i(\Phi):=\sup\{q_1\geq0:\Phi\text{ is of lower type }q_1\}$$
and
$$I(\Phi):=\inf\{q_2\geq0:\Phi\text{ is of upper type }q_2\}.$$
\item [(iii)] {\rm (Growth function)} A function $\Phi:[0,\infty)\to[0,\infty)$
is a growth function if $\Phi$ is of upper type $1$ and of lower type $q$ for
some $q\in(0,1]$.
\item[(iv)] {\rm (Orlicz space)} Let $\Phi:[0,\infty)\to[0,\infty)$ be an
Orlicz function. The Orlicz space $L^\Phi(\mathbb{R}^n)$ is defined as
the space of all Lebesgue measurable functions $f$ on $\mathbb{R}^n$ such
that
$$\|f\|_{L^\Phi(\mathbb{R}^n)}:=\inf\Big\{\lambda>0:\int_{\mathbb{R}^n}\Phi\Big(\frac{|f(x)|}{\lambda}\Big)\,dx\leq1\Big\}<\infty.$$
\item[(v)] {\rm ($\mathrm{BMO}_{\Phi}(\mathbb{R}^n)$ space)} A locally integrable
function $b$ is said to belong to $\mathrm{BMO}_{\Phi}(\mathbb{R}^n)$ if
$$\|b\|_{\mathrm{BMO}_{\Phi}(\mathbb{R}^n)}:=\sup_{B\subset\mathbb{R}^n}\frac{1}{\|\chi_B\|_{L^\Phi(\mathbb{R}^n)}}\int_B|b(x)-b_B|\,dx<\infty,$$
where the supremum is taken over all balls $B\subset\mathbb{R}^n$.
\end{itemize}
\end{definition}

We next  introduce the Orlicz--Hardy space and the $b$-adapted Orlicz--Hardy space.

\begin{definition}\label{def3.2}
\begin{itemize}
\item[(i)] {\rm (Orlicz--Hardy space)} For $f\in\mathcal{S}'(\mathbb{R}^n)$,
we define the nontangential grand maximal function by
$$\mathcal{M}f(x):=\sup_{\phi\in\mathcal{S}_0(\mathbb{R}^n)}\sup_{\substack{t>0\\ |y-x|<t}}|f*\phi_t(y)|,$$
where $\phi_t(x):=t^{-n}\phi(x/t)$ and
$$\mathcal{S}_0(\mathbb{R}^n):=\Big\{\phi\in\mathcal{S}(\mathbb{R}^n):\sup_{\substack{x\in\mathbb{R}^n\\ |\beta|\leq1}}(1+|x|)^{n+1}|\partial^\beta\phi(x)|\leq1\Big\}.$$
The Orlicz--Hardy space $H^\Phi(\mathbb{R}^n)$ is defined as the space of
all $f\in\mathcal{S}'(\mathbb{R}^n)$ such that $\mathcal{M}f\in L^\Phi
(\mathbb{R}^n)$, equipped with the quasi-norm
$$\|f\|_{H^\Phi(\mathbb{R}^n)}:=\|\mathcal{M}f\|_{L^\Phi(\mathbb{R}^n)}.$$
\item[(ii)] {\rm ($b$-adapted Orlicz--Hardy space)} Let $\Phi$ be a growth
function satisfying $0<i(\Phi)\leq1$, and let $b\in\mathrm{BMO}_{\Phi}(\mathbb{R}^n)$.
The $b$-adapted Orlicz--Hardy space $H_b^\Phi(\mathbb{R}^n)$ is defined as
the collection of all $f\in H^\Phi(\mathbb{R}^n)$ such that
$$[b,\mathcal M]f(x):=\mathcal M\big((b(x)-b(\cdot))f(\cdot)\big)(x)\in L^1(\mathbb{R}^n).$$
We equip $H_b^\Phi(\mathbb{R}^n)$ with the norm
$$\|f\|_{H_b^\Phi(\mathbb{R}^n)}:=\|f\|_{H^\Phi(\mathbb{R}^n)}\|b\|_{\mathrm{BMO}_{\Phi}(\mathbb{R}^n)}+\|[b,\mathcal{M}]f\|_{L^1(\mathbb{R}^n)}.$$
\end{itemize}
\end{definition}

Before presenting the definition of Orlicz--Hardy atom, let us recall one
definition.
\begin{definition}\label{def3.3}
{\rm (The spaces $L_B^q$)} Let $B\subset\mathbb{R}^n$ and $q\in[1,\infty]$.
We denote by $L_B^q$ the space of all measurable functions $f$ on $\mathbb{R}^n$
supported in $B$ such that
$$\|f\|_{L_B^q}:=
\begin{cases}
\displaystyle\Big(\frac{1}{|B|}\int_B|f(x)|^qdx\Big)^{1/q}<\infty, & 1\leq q<\infty,\\[2mm]
\displaystyle\operatorname*{ess\,sup}_{x\in B}|f(x)|<\infty, & q=\infty.
\end{cases}$$
\end{definition}

We now define Orlicz--Hardy atoms.

\begin{definition}\label{def3.4}
{\rm (Orlicz--Hardy atom)} Let $\Phi$ be a growth function with
critical lower type index $i(\Phi)$, and set
$m(\Phi):=[n(i(\Phi)^{-1}-1)]$. Let $q\in(1,\infty]$, and let $s$ be a
nonnegative integer satisfying $s\geq m(\Phi)$. A locally integrable
function $a$ on $\mathbb{R}^n$ is called a $(\Phi,q,s)$-atom associated
with a ball $B\subset\mathbb{R}^n$ if the following statements hold:
\begin{itemize}
\item[(i)] $\operatorname{supp}a\subset B$;
\item[(ii)] $\|a\|_{L_B^q}\leq\|\chi_B\|_{L^\Phi(\mathbb{R}^n)}^{-1}$;
\item [(iii)] $\int_{\mathbb{R}^n}x^\beta a(x)\,dx=0$ for every multi-index
$\beta\in\mathbb{N}^n$ with $|\beta|\leq s$.
\end{itemize}
\end{definition}
\begin{remark}\label{rem3.1}
By the assumptions  of Theorem \ref{thm1.2}, i.e., $n/(n+\alpha)<i(\Phi)\leq1$
with $0<\alpha\leq1$, we have $m(\Phi)=0$. Hence, the atoms used in the
proof of Theorem \ref{thm1.2} need to satisfy the cancellation condition
$\int_{\mathbb{R}^n}a(x)dx=0$.
\end{remark}

Now we introduce the following class of sublinear operators, which is
very useful in our proof.

\begin{definition}\label{def3.5}
Let $\Phi$ be a growth function with $0<i(\Phi)\leq1$. We say that a
sublinear operator $\mathfrak{T}$ belongs to $\mathcal{K}_{\Phi}$ if
the following conditions are satisfied:
\begin{itemize}
\item[(i)] $\mathfrak{T}$ is bounded from $H^1(\mathbb{R}^n)$ to
$L^1(\mathbb{R}^n)$;
\item[(ii)] $\mathfrak{T}$ is bounded from $L^1(\mathbb{R}^n)$ to
$L^{1,\infty}(\mathbb{R}^n)$;
\item[(iii)] for every $b\in\mathrm{BMO}_{\Phi}(\mathbb{R}^n)$ and
every $(\Phi,q,0)$-atom $a$ associated with a ball $B\subset\mathbb{R}^n$,
where $q\in(1,\infty)$,
$$\|(b-b_B)\mathfrak{T}a\|_{L^1(\mathbb{R}^n)}
\lesssim\|a\|_{L_B^q}\|\chi_B\|_{L^\Phi(\mathbb{R}^n)}\|b\|_{\mathrm{BMO}_{\Phi}(\mathbb{R}^n)}.$$
\end{itemize}
\end{definition}

\subsection{Proof of Theorem \ref{thm1.2}}
This subsection aims to prove Theorem \ref{thm1.2}.
Before presenting the proof of Theorem \ref{thm1.2}, let us
establish some preliminary lemmas. To obtain the endpoint estimates
for commutators associated with operators in $\mathcal{K}_{\Phi}$, we
recall the following auxiliary result based on the bilinear decomposition
of products of Orlicz--Hardy and Orlicz--BMO functions.

\begin{lemma}\label{lem3.1}{\rm (\cite{FL})}
Let $\Phi$ be a growth function satisfying condition (1.1) and
$$\frac{n}{n+1}<i(\Phi)\leq I(\Phi)<1.$$
Let $b\in\mathrm{BMO}_{\Phi}(\mathbb{R}^n)$.
Set
$$\Pi_4(f,b):=\sum_{I\in\mathcal{D}}\sum_{\lambda\in E}\langle f,\psi_I^\lambda\rangle\langle b,\psi_I^\lambda\rangle(\psi_I^\lambda)^2,$$
where $\mathcal{D}$ denotes the collection of dyadic cubes,
$E:=\{0,1\}^n\setminus\{(0,\ldots,0)\}$, and
$\{\psi_I^\lambda\}_{I\in\mathcal{D},\,\lambda\in E}$ is the wavelet family.
Then the following conclusions hold:
\begin{itemize}
\item[(i)] The operator $\Pi_4$ extends to a bounded bilinear operator
from $H^\Phi(\mathbb{R}^n)\times\mathrm{BMO}_{\Phi}(\mathbb{R}^n)$ to
$L^1(\mathbb{R}^n)$. In particular,
$$\|\Pi_4(f,b)\|_{L^1(\mathbb{R}^n)}\lesssim\|f\|_{H^\Phi(\mathbb{R}^n)}\|b\|_{\mathrm{BMO}_{\Phi}(\mathbb{R}^n)}.$$
\item[(ii)] Let $T\in K_\Phi$. For $f\in H^\Phi(\mathbb R^n)$, define
$T_bf(x):=T\big((b(x)-b(\cdot))f(\cdot)\big)(x).$ Then there exists a nonnegative operator
$$\mathcal{R}:H^\Phi(\mathbb{R}^n)\times\mathrm{BMO}_{\Phi}(\mathbb{R}^n)\longrightarrow L^1(\mathbb{R}^n)$$
such that, for every $f\in H^\Phi(\mathbb{R}^n)$,
$$|T_b f(x)|\leq |T(\Pi_4(f,b))(x)|+\mathcal{R}(f,b)(x)$$
for almost every $x\in\mathbb{R}^n$, and
$$\|\mathcal{R}(f,b)\|_{L^1(\mathbb{R}^n)}\lesssim\|f\|_{H^\Phi(\mathbb{R}^n)}\|b\|_{\mathrm{BMO}_{\Phi}(\mathbb{R}^n)}.$$
\item[(iii)] For $f\in H^\Phi(\mathbb{R}^n)$,
$$f\in H_b^\Phi(\mathbb{R}^n)\Longleftrightarrow\Pi_4(f,b)\in H^1(\mathbb{R}^n).$$
Moreover, whenever these equivalent conditions hold,
$$\|f\|_{H_b^\Phi(\mathbb{R}^n)}\simeq\|f\|_{H^\Phi(\mathbb{R}^n)}\|b\|_{\mathrm{BMO}_{\Phi}(\mathbb{R}^n)}+\|\Pi_4(f,b)\|_{H^1(\mathbb{R}^n)}.$$
\end{itemize}
\end{lemma}

In order to prove Theorem \ref{thm1.2}, we need to verify that
the intrinsic Littlewood--Paley $g$-function and the intrinsic $g_{\lambda,
\alpha}^{*}$ belong to the class $\mathcal{K}_{\Phi}$.

\begin{lemma}\label{lem3.2}
Let $0<\alpha\leq1$, and let $\Phi$ be a growth function satisfying
$n/(n+\alpha)<i(\Phi)\leq1$. Then $g_{\alpha}\in\mathcal{K}_{\Phi}$.
\end{lemma}
\begin{proof}
It was proved in \cite{W,HL} that $g_{\alpha}$ is bounded from
$H^1(\mathbb{R}^n)$ to $L^1(\mathbb{R}^n)$ and from $L^1(\mathbb{R}^n)$
to $L^{1,\infty}(\mathbb{R}^n)$. Therefore, we only verify condition
{\rm (iii)} in the definition of $\mathcal{K}_{\Phi}$.

Let $b\in\mathrm{BMO}_{\Phi}(\mathbb{R}^n)$ and let $a$ be a $(\Phi,q,
0)$-atom supported in a ball $B:=B(x_0,r)$, where $q\in(1,\infty)$.
Write
$$\begin{array}{ll}
&\|(b-b_B)g_{\alpha}(a)\|_{L^1(\mathbb{R}^n)}=\displaystyle\int_{2B}|b(x)-b_B|\,g_{\alpha}(a)(x)dx
+\int_{(2B)^c}|b(x)-b_B|\,g_{\alpha}(a)(x)dx\\
&\qquad\qquad\qquad\qquad\qquad=:I_1+I_2.
\end{array}$$
For $I_1$, by H\"older's inequality and the $L^q(\mathbb{R}^n)$-boundedness
of $g_{\alpha}$, one gets
$$\begin{array}{ll}
&I_1\leq\displaystyle\Big(\int_{2B}|b(x)-b_B|^{q'}dx\Big)^{1/q'}\Big(\int_{2B}|g_{\alpha}(a)(x)|^qdx\Big)^{1/q}\\
&\quad\lesssim |B|^{1/q'-1}\|\chi_B\|_{L^\Phi(\mathbb{R}^n)}\|b\|_{\mathrm{BMO}_{\Phi}(\mathbb{R}^n)}|B|^{1/q}\|a\|_{L_B^q}\\
&\quad=\|a\|_{L_B^q}\|\chi_B\|_{L^\Phi(\mathbb{R}^n)}\|b\|_{\mathrm{BMO}_{\Phi}(\mathbb{R}^n)}.
\end{array}$$

We next estimate $I_2$. Arguing as in the proof of Lemma~\ref{lem2.7},
and using $\int_B|a(y)|\,dy\lesssim |B|\|a\|_{L_B^q}$, we obtain
$$g_{\alpha}(a)(x)\lesssim 2^{-j(n+\alpha)}\|a\|_{L_B^q},\qquad x\in 2^{j+1}B\setminus 2^jB,\quad j\in\mathbb{N}\setminus\{0\}.$$
Choose $\mathfrak{p}$ such that $\Phi$ is of lower type $\mathfrak{p}$ and
$n/(n+\alpha)<\mathfrak{p}<i(\Phi)$. Then
$$\begin{array}{ll}
&I_2\leq\displaystyle\sum\limits_{j=1}^{\infty}\int_{2^{j+1}B\setminus2^jB}|b(x)-b_B|\,g_{\alpha}(a)(x)dx\\
&\quad\lesssim\displaystyle\|a\|_{L_B^q}\sum_{j=1}^{\infty}2^{-j(n+\alpha)}\int_{2^{j+1}B}|b(x)-b_B|dx\\
&\quad\lesssim\displaystyle\|a\|_{L_B^q}\|\chi_B\|_{L^\Phi(\mathbb{R}^n)}\|b\|_{\mathrm{BMO}_{\Phi}(\mathbb{R}^n)}\sum\limits_{j=1}^{\infty}2^{-j(n+\alpha)}2^{(j+1)n/\mathfrak{p}}\\
&\quad\lesssim \|a\|_{L_B^q}\|\chi_B\|_{L^\Phi(\mathbb{R}^n)}\|b\|_{\mathrm{BMO}_{\Phi}(\mathbb{R}^n)},
\end{array}$$
since $\mathfrak{p}>n/(n+\alpha)$. In view of the above estimates for
$I_1$ and $I_2$, we obtain
$$\|(b-b_B)g_{\alpha}(a)\|_{L^1(\mathbb{R}^n)}\lesssim\|a\|_{L_B^q}\|\chi_B\|_{L^\Phi(\mathbb{R}^n)}
\|b\|_{\mathrm{BMO}_{\Phi}(\mathbb{R}^n)}.$$
Thus, $g_{\alpha}$ satisfies condition {\rm (iii)} in the definition
of $\mathcal{K}_{\Phi}$. Together with its boundedness from $H^1(\mathbb{R}^n)$
to $L^1(\mathbb{R}^n)$ and its weak type $(1,1)$ boundedness, we conclude that
$g_{\alpha}\in\mathcal{K}_{\Phi}$. This completes the proof.
\end{proof}
\begin{lemma}\label{lem3.3}
Let $0<\alpha\leq 1$, $2+2\alpha/n<\lambda<\infty$, and let $\Phi$ be a growth
function satisfying $\frac{n}{n+\alpha}<i(\Phi)\leq I(\Phi)<1$. Then
the intrinsic $g_{\lambda,\alpha}^*$-function $g_{\lambda,\alpha}^*
\in\mathcal{K}_{\Phi}$.
\end{lemma}
\begin{proof}
The proof is similar to that of Lemma \ref{lem3.2}. Under the assumption
$\lambda>2+\frac{2\alpha}{n}$, it was shown in \cite{LY,W} that the
intrinsic Littlewood--Paley $g_{\lambda,\alpha}^{*}$-function is bounded
from $H^1(\mathbb{R}^n)$ to $L^1(\mathbb{R}^n)$ \cite{LY} and is bounded
on $L^q(\mathbb{R}^n)$ for every $q\in(1,\infty)$. It is worth noting that
the Calder\'on--Zygmund decomposition argument used in \cite{Wang2013} for
the weak type $(1,1)$ estimate, originally carried out under the condition
$\lambda>3+\frac{2\alpha}{n}$, can be modified, by using the direct far-field
argument developed in Section \ref{S2}, to yield the same estimate whenever
$\lambda>2+\frac{2\alpha}{n}$. Thus, $g_{\lambda,\alpha}^{*}$ satisfies
conditions {\rm (i)} and {\rm (ii)} in the definition of $\mathcal{K}_{\Phi}$.

It remains only to verify condition {\rm (iii)}. Let $b\in\mathrm{BMO}_{\Phi}
(\mathbb{R}^n)$ and let $a$ be a $(\Phi,q,0)$-atom supported in a ball $B$.
Repeating the local estimate in the proof of Lemma \ref{lem3.2} and using the
far-field estimate for $g_{\lambda,\alpha}^{*}(a)$ established in Section
\ref{S2}, we obtain
$$\|(b-b_B)g_{\lambda,\alpha}^{*}(a)\|_{L^1(\mathbb{R}^n)}\lesssim
\|a\|_{L_B^q}\|\chi_B\|_{L^\Phi(\mathbb{R}^n)}\|b\|_{\mathrm{BMO}_{\Phi}(\mathbb{R}^n)}.$$
Hence $g_{\lambda,\alpha}^{*}$ satisfies condition {\rm (iii)}, and therefore
$g_{\lambda,\alpha}^{*}\in\mathcal{K}_{\Phi}$.
\end{proof}

Now we provide a proof of Theorem \ref{thm1.2}.

\begin{proof}[\bf Proof of Theorem \ref{thm1.2}]
By Lemma \ref{lem3.2} we see that $g_\alpha\in\mathcal{K}_\Phi$. So Lemma
\ref{lem3.1} is applicable with $T=g_\alpha$.

{\it Proof of part (i)}. Let $f\in H_b^\Phi(\mathbb{R}^n)$. By Lemma \ref{lem3.1}(ii),
$$|g_{\alpha,b}(f)(x)|\lesssim g_\alpha(\Pi_4(f,b))(x)+\mathcal{R}(f,b)(x)$$
for almost every $x\in\mathbb{R}^n$. Moreover, Lemma~\ref{lem3.1}(iii) gives
$\Pi_4(f,b)\in H^1(\mathbb{R}^n)$ and
$$\|f\|_{H_b^\Phi(\mathbb{R}^n)}\sim \|f\|_{H^\Phi(\mathbb{R}^n)}\|b\|_{\mathrm{BMO}_{\Phi}(\mathbb{R}^n)}+\|\Pi_4(f,b)\|_{H^1(\mathbb{R}^n)}.$$
In view of the boundedness of $g_\alpha$ from $H^1(\mathbb{R}^n)$ to
$L^1(\mathbb{R}^n)$ and Lemma \ref{lem3.1}(ii), one gets
$$\begin{array}{ll}
&\|g_{\alpha,b}(f)\|_{L^1(\mathbb{R}^n)}\lesssim \|g_\alpha(\Pi_4(f,b))\|_{L^1(\mathbb{R}^n)}+\|\mathcal{R}(f,b)\|_{L^1(\mathbb{R}^n)}\\
&\qquad\qquad\qquad\quad\lesssim \|\Pi_4(f,b)\|_{H^1(\mathbb{R}^n)}+\|f\|_{H^\Phi(\mathbb{R}^n)}\|b\|_{\mathrm{BMO}_{\Phi}(\mathbb{R}^n)}\\
&\qquad\qquad\qquad\quad\lesssim \|f\|_{H_b^\Phi(\mathbb{R}^n)}.
\end{array}$$
This shows that $g_{\alpha,b}$ is bounded from $H_b^\Phi(\mathbb{R}^n)$ to
$L^1(\mathbb{R}^n)$.

{\it Proof of part (ii)}. Let $f\in H^\Phi(\mathbb{R}^n)$. Since $g_\alpha$
is of weak type $(1,\,1)$, Lemma \ref{lem3.1} (i) and (ii) yield
$$\begin{array}{ll}
&\|g_{\alpha,b}(f)\|_{L^{1,\infty}(\mathbb{R}^n)}\lesssim \|g_\alpha(\Pi_4(f,b))\|_{L^{1,\infty}(\mathbb{R}^n)}+\|\mathcal{R}(f,b)\|_{L^{1,\infty}(\mathbb{R}^n)}\\
&\qquad\qquad\qquad\quad\lesssim \|\Pi_4(f,b)\|_{L^1(\mathbb{R}^n)}+\|\mathcal{R}(f,b)\|_{L^1(\mathbb{R}^n)}\\
&\qquad\qquad\qquad\quad\lesssim \|f\|_{H^\Phi(\mathbb{R}^n)}\|b\|_{\mathrm{BMO}_{\Phi}(\mathbb{R}^n)}.
\end{array}$$
Therefore, $g_{\alpha,b}$ is bounded from $H^\Phi(\mathbb{R}^n)$ to
$L^{1,\infty}(\mathbb{R}^n)$. This completes the proof of part (ii).

{\it Proofs of parts (iii) and (iv)}. By the pointwise comparability
between $S_{\alpha}$ and $g_{\alpha}$, parts (iii) and (iv) follow
directly from parts (i) and (ii).

{\it Proofs of parts (v) and (vi)}. By Lemma \ref{lem3.3}, we see that
$g_{\lambda,\alpha}^{*}\in\mathcal{K}_{\Phi}$. The desired conclusions
follow from Lemma \ref{lem3.1} and the arguments similar to those used
to derive parts (i) and (ii). The details are omitted.
\end{proof}

\medskip
\noindent{\bf Acknowledgments.}
Yanyan Han is supported by the National Natural Science Foundation of China (grant No. 12526513), and the Open Research Fund of Hubei Key Laboratory of Mathematical
Sciences (Central China Normal University), Wuhan 430079, P. R. China. Feng Liu
is supported by the Natural Science Foundation of Shandong Province (grant No.
ZR2023MA022) and the National Natural Science Foundation of China (grant No. 12326371). Yongming Wen is supported by Fujian Provincial Natural Science Foundation of China (grant No. 2025J01361).
Huoxiong Wu is supported by the National Natural Science Foundation of China (grant
No. 12271041).

\medskip

\noindent {\bf Compliance with Ethical Standard}

\medskip

\noindent {Conflict of Interests:} The authors declare that there is no conflict of interests regarding the publication of this paper.

\medskip

\noindent {Data Availability Statements:} No datasets were generated or analysed during the current study.

\end{document}